\documentclass[12pt]{amsart}

\usepackage[english]{babel}
\usepackage[utf8]{inputenc}
\usepackage{amsmath}
\usepackage{amssymb}
\usepackage{amsfonts}
\usepackage{amsthm}
\usepackage{mathrsfs}
\usepackage[all]{xy}
\usepackage[pdftex]{graphicx}
\usepackage{color}
\usepackage{cite}
\usepackage{url}
\usepackage{indent first}
\usepackage[labelfont=bf,labelsep=period,justification=raggedright]{caption}
\usepackage[english]{babel}
\usepackage[utf8]{inputenc}
\usepackage{hyperref}
\usepackage[colorinlistoftodos]{todonotes}
\usepackage{tikz}

\newcommand{\NN}{\mathbb{N}}

\newcommand{\QQ}{\mathbb{Q}}

\newcommand{\ZZ}{\mathbb{Z}}

\theoremstyle{plain}
\newtheorem{thm}{Theorem}

\theoremstyle{definition}
\newtheorem{defn}[thm]{Definition}

\theoremstyle{remark}
\newtheorem{rem}[thm]{Remark}

\numberwithin{equation}{section}
\numberwithin{thm}{section}

\begin{document}

\title{$N$-Division Points of Hypocycloids} \label{hypocycloids}

\author{Nitya Mani}
\address{The Harker School}
\email{15nityam@students.harker.org}

\author{Simon Rubinstein-Salzedo}
\address{Department of Statistics, 390 Serra Mall, Stanford University, Stanford, CA 94305}
\email{simonr@stanford.edu}
\date{\today}

\begin{abstract} 

We show that the $n$-division points of all rational hypocycloids are constructible with an unmarked straightedge and compass for all integers $n$, given a pre-drawn hypocycloid. We also consider the question of constructibility of $n$-division points of hypocycloids without a pre-drawn hypocycloid in the case of a tricuspoid, concluding that only the $1$, $2$, $3$, and $6$-division points of a tricuspoid are constructible in this manner.
\end{abstract}

\maketitle

\section{Introduction} \label{Introduction}

One of the oldest classes of problems in mathematics is concerned with straightedge and compass constructions. Most famous in this collection are the three classical Greek geometrical challenges: trisecting an angle, squaring a circle, and duplicating the cube (i.e.\ constructing a segment of length $\sqrt[3]{2}$) with an unmarked straightedge and compass, all of which have been shown to be impossible. The first and third of these are elementary exercises in field theory, while the second requires the nontrivial fact that $\pi$ is transcendental (or at least some similarly nontrivial information about $\pi$).

However, straightedge and compass constructions have by no means been completely resolved by the work of the ancient Greeks or later mathematicians. One of the most interesting straightedge and compass problems asks which regular polygons can be constructed with the aforementioned two tools. In 1796, Carl Friedrich Gauss demonstrated that the $17$-gon was constructible with straightedge and compass, leading him in 1801 to arrive at a general, sufficient condition for polygon constructibility. Pierre Wantzel showed the necessity of this condition in 1837, leading to the final statement of the Gauss-Wantzel Theorem.

\begin{thm}[Gauss-Wantzel]\label{gw} A regular $n$-gon can be constructed with a straightedge and compass if and only if $n$ is of the form \begin{equation} n = 2^m p_1 \ldots p_s \label{constructible} \end{equation} for distinct Fermat primes $p_i.$ (A Fermat prime is a prime $p$ of the form $p = 2^{2^a} + 1,\, a \in \NN$.) \end{thm}

A reinterpretation of the Gauss-Wantzel Theorem is that it is possible to divide a circle into $n$ arcs of equal length with a straightedge and compass if and only if $n$ is of the form (\ref{constructible}). It is natural to ask the corresponding question for other closed curves.  


\begin{defn} The $n$-division points of a closed curve $C$ are a set of $n$ points on $C$ that divide $C$ into pieces of equal arclength. \end{defn}


There is one other classical theorem similar to the Gauss-Wantzel Theorem: Abel's Theorem on the Lemniscate.

\begin{thm}[Abel]\label{atl} The $n$-division points of the lemniscate defined by $(x^2+y^2)^2=x^2-y^2$ can be constructed with a straightedge and compass if and only if $n$ is of the form \[n = 2^m p_1 \ldots p_s\] for distinct Fermat primes $p_i.$ \end{thm}

Detailed expositions of this theorem can be found in~\cite{COX12} and~\cite{COX14}.

Surprisingly, Theorem~\ref{atl} arrives at the same closed-form expression~(\ref{constructible}) for the constructible $n$-division points of the lemniscate as Gauss and Wantzel did for the constructible regular $n$-gons.

More recently, Cox and Shurman in~\cite{COX05} have investigated the $n$-division points of other closed curves, including two members of a class of clover-like curves (the cardioid and a specific three-leaved clover) that include the lemniscate and circle. We continue the study of $n$-division points on curves by investigating the class of rational hypocycloids, defined in \S\ref{Hypocycloid}.

\begin{rem}\label{remark1} For curves other than the circle, there are actually \emph{two} possible problems regarding the constructibility of $n$-division points. First, for a closed curve $C$, we can ask whether, given the curve $C$ already drawn on the plane, it is possible to construct the $n$-division points using a compass and a straightedge. However, we can also ask whether it is possible to construct the $n$-division points of $C$ \emph{without} $C$ ever being drawn. To us, the former question feels more natural: it seems like a cruel prank to construct the $n$-division points of $C$ without ever getting to see them proudly displayed on the curve to which they rightfully belong. However, we will also consider the second problem where we do not have a pre-drawn figure, in the case of the tricuspoid. We will consider the first of these questions in \S\ref{Tricuspoid} and the second in \S\ref{notricuspoid}. Abel's Theorem is an answer to a problem of the second type. To the best of our knowledge, the corresponding question with a drawn lemniscate has not yet been answered. \end{rem}


\section{The $a/b$ Hypocycloid} \label{Hypocycloid}

The hypocycloid is a curve obtained by rolling a circle of radius $b$ inside the circumference of a larger circle of radius $a$ and tracing the path of a fixed point $P$ on the smaller circle of radius $b$ (a visual depiction of this rolling can be found in Figure~\ref{fig:tricuspoid} for the tricuspoid where $a/b=3$). 
These curves can be represented in parametric form as \begin{align*} x(\phi) &= (a-b)\cos\phi+b\cos{\left(\frac{a-b}{b}\phi \right)}, \\ y(\phi) &= (a-b)\sin\phi-b\sin{\left(\frac{a-b}{b}\phi \right)}.\end{align*} Here, $\phi$ denotes the angle between the positive $x$-axis and the segment connecting the centers of the two circles, as in shown in Figure~\ref{fig:phi}.

\begin{figure}
\includegraphics{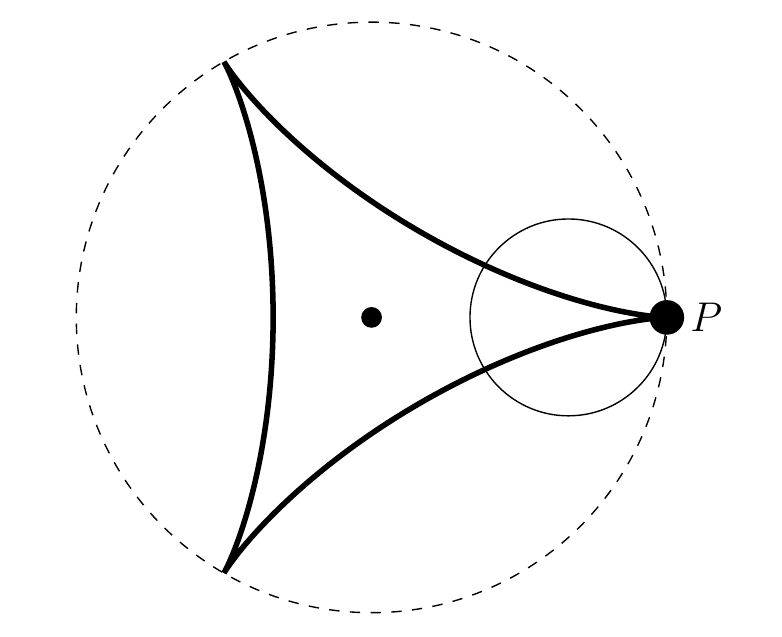}
\caption{The tricuspoid curve is the hypocycloid with three cusps obtained by rolling a circle of radius $b$ (the small circle) counterclockwise inside a circle of radius $a=3b$ (the dashed outer circle) and tracing the path of a point (the large marked point labeled $P$) on the smaller circle.}
\label{fig:tricuspoid}
\end{figure}


It is natural to think about integer hypocycloids, where $a/b = c \in \ZZ_{>1}$, which have $c$ cusps. However, none of our arguments are specific to the integral case, except in \S\ref{notricuspoid}, where we assume that $c=3$. The geometric picture is slightly different in the case that $c$ is not an integer: if $c=a/b$ where $(a,b)=1$, then the $c$-hypocycloid will have $a$ cusps, and the small circle must travel around the large circle $b$ times before the marked point completes its trajectory.

\begin{rem} From now on, we will consider the scaled hypocycloid, in which $a = c$ and $b = 1$. \end{rem}

Note that it is necessary that $c>1$, and when $c=2$, the $c$-hypocycloid is simply the line segment $[-2,2]$, traced out twice.

We characterize the constructible $n$-division points for all hypocycloids where $a/b = c \in \QQ$. Of particular mathematical importance among this family of hypocycloids are the hypocycloids with $c=3$ and $4$, called respectively the tricuspoid (also known as the deltoid) and the astroid. The tricuspoid arises naturally in many areas of mathematics, and it was studied in depth by Euler and Steiner. For some more recent appearances, see~\cite{DUN09} and~\cite{SB12}.
The astroid is also mathematically noteworthy in its relation to the tricuspoid, and because it is the envelope of a set of segments of constant length whose ends can be found on mutually perpendicular straight lines (see~\cite{HYPWolfram}). Several examples of hypocycloids are illustrated in Figure~\ref{fig:Hypocycloid}).

\begin{figure}
\includegraphics{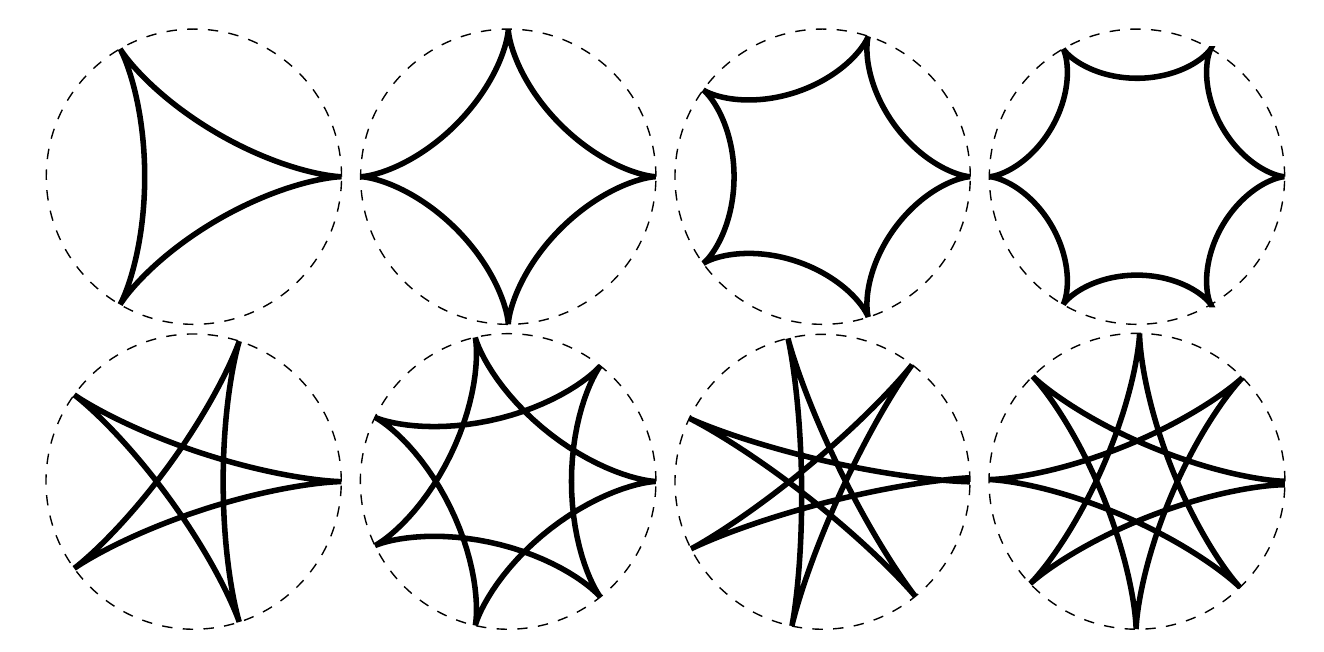}
\caption{The above diagram illustrates several examples of hypocycloids obtained by varying the value of $c$. The top row shows integral hypocycloids, including the tricuspoid and astroid, and the bottom row provides some examples of nonintegral rational hypocycloids, with $c=5/2,7/2,7/3,8/3$, respectively.}
\label{fig:Hypocycloid}
\end{figure}

\section{$n$-Division Points of the $c$-Hypocycloid} \label{Tricuspoid}

\begin{defn} If $C$ is a closed curve (possibly with self-intersections, but implicitly considered with a parametrization), $P$ a point on $C$, and $n$ a positive integer, then the $n$-division points of $C$ relative to $P$ are the $n$ points $P=P_1,P_2,\ldots,P_n$ so that the arclength along $C$ from $P_i$ to $P_{i+1}$ and from $P_n$ to $P_1$ are all equal. \end{defn}

In the case of the $c$-hypocycloid, we will implicitly assume that $P=(c,0)$, the rightmost point on the $x$-axis.

We say that the $n$-division points of $C$ are constructible with compass and straightedge if \emph{all} of them can be constructed with a compass and straightedge.

\begin{thm} For all rational $c>1$ and positive integers $n$, the $n$-division points of the $c$-hypocycloid are constructible with a straightedge and compass when a $c$-hypocycloid has been drawn in advance. \end{thm}

\begin{proof}
The $c$-cusped hypocycloid can be parametrized in terms of the tangential angle $\phi$ with respect to the $x$ axis, as shown in Figure~\ref{fig:phi}, to arrive at the following rectangular parametric equations: \begin{align*} x(\phi) &= (c-1)\cos{\phi}+\cos{((c-1)\phi)}, \\ y(\phi) &= (c-1)\sin{\phi}-\sin{((c-1)\phi)}.\end{align*}

\begin{figure}
\includegraphics{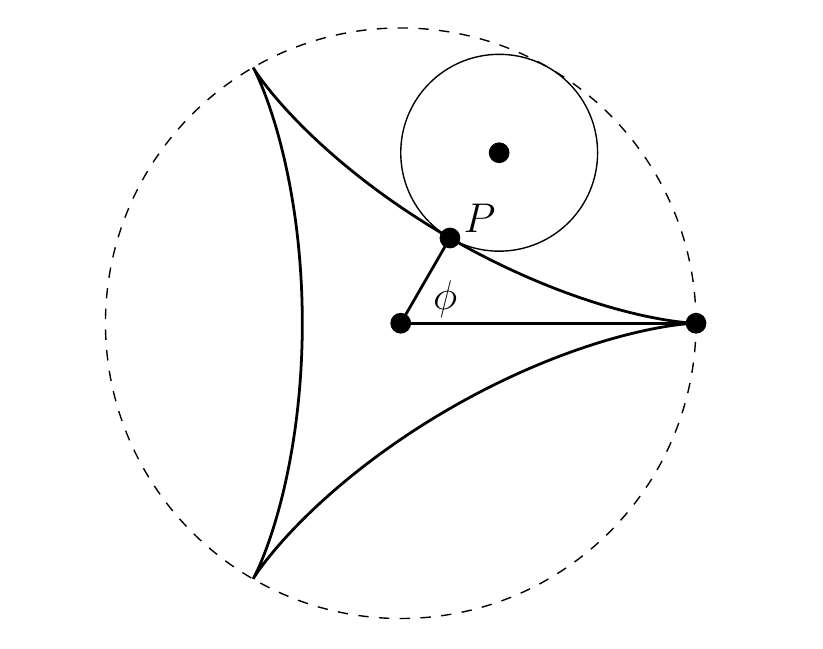}
\caption{Hypocycloids are parametrized in terms of the tangential angle $\phi$, the angle formed by the intersection of the $x$ axis and the line connecting the centers of the two circles, or equivalently to the segment connecting the origin and $P$. If $c\in\NN$, then as $\phi$ goes from $0$ to $2\pi$, the entire hypocycloid is traced out at least once, and more generally as $\phi$ goes from $0$ to $2\pi b$, where $b$ is the denominator of $c$ in lowest terms.}
\label{fig:phi}
\end{figure}

With this parametrization, we can obtain equations for the polar radius $r$ and arclength $s$ as a function of $\phi$ using the aforementioned rectangular equations and a line integral: the polar radius is given by \[r(\phi) = \sqrt{x(\phi)^2+y(\phi)^2}=\sqrt{c^2-2c+2 + 2(c-1)\cos{(c\phi)}},\] and the arclength by \[s(\phi) = \int_0^\phi r(\psi)\, d\psi = \frac{8(c-1)}{c}\sin ^2{\left(\frac{c\phi}{4}\right)}.\] We write $\varpi_c = 8b(c-1)$ for the total arclength of the hypocycloid, obtained by evaluating the arclength line integral from $0$ to $2\pi$.

Using this information, we can express the polar radius $r$ in terms of the arclength $s$, assuming that the tangential angle $\phi$ lies in the range $[0,2\pi/c]$ (or, the piece of the hypocycloid up to the first cusp): \begin{equation} \label{polarradius} r =\sqrt{c^2-2c+2 + 2(c-1)\cos{\left(2\cos ^{-1}{\left(1+\frac{cs}{4-4c}\right)}\right)}}.\end{equation}

This given expression can be used to determine the radius of a circle that, when intersected with a drawn $c$-cusped hypocycloid, creates an arc on the hypocycloid of some length $\frac{\varpi_c}{n}$, where $\varpi_c$ is the total arclength of the $c$-hypocycloid. That is, we input some arclength $s=\frac{\varpi_c}{n}$ into (\ref{polarradius}). We obtain a polar radius $r$, and when we construct a circle centered at the origin of radius $r$, one of its intersections with the $c$-cusped hypocycloid in question will cut off a portion of the arclength of length $s$ ($1/n^\text{th}$ of the total arclength), meaning that the point of intersection is one of the $n$-division points of the curve. While this argument only shows that the \emph{first} $n$-division point is constructible, the argument also holds if we replace $n$ above by $n/d$, which gives us the $d^\text{th}$ $n$-division point. Thus, we know that the arclengths formed by finding the $n$-division points of the $c$-cusped hypocycloid for all positive integers $n$, are always rational. 

Simplifying our expression for polar radius $r$, we obtain \[r=\sqrt{c^2-2c+2(c-1)\left(2\left(1+\frac{cs}{4-4c} \right)^2-1 \right)}.\] We can always express the polar radius $r$ in terms of the arclength  $s$ we wish to isolate using only rational expressions and square roots, meaning that $r$ lies in a quadratic extension of $\QQ$ and is thus constructible with compass and straightedge. Hence, when we start with a $c$-cusped hypocycloid already drawn, we can construct circles of suitable radii $r$ and intersect them with the hypocycloid in order to construct points that divide the hypocycloid into $n$ pieces of equal arclength for all integers $n$. 
Thus, for all $c$, we can construct the $n$-division points of the $c$-cusped hypocycloid for all positive integers $n$. 
\end{proof}

\begin{figure}
\includegraphics{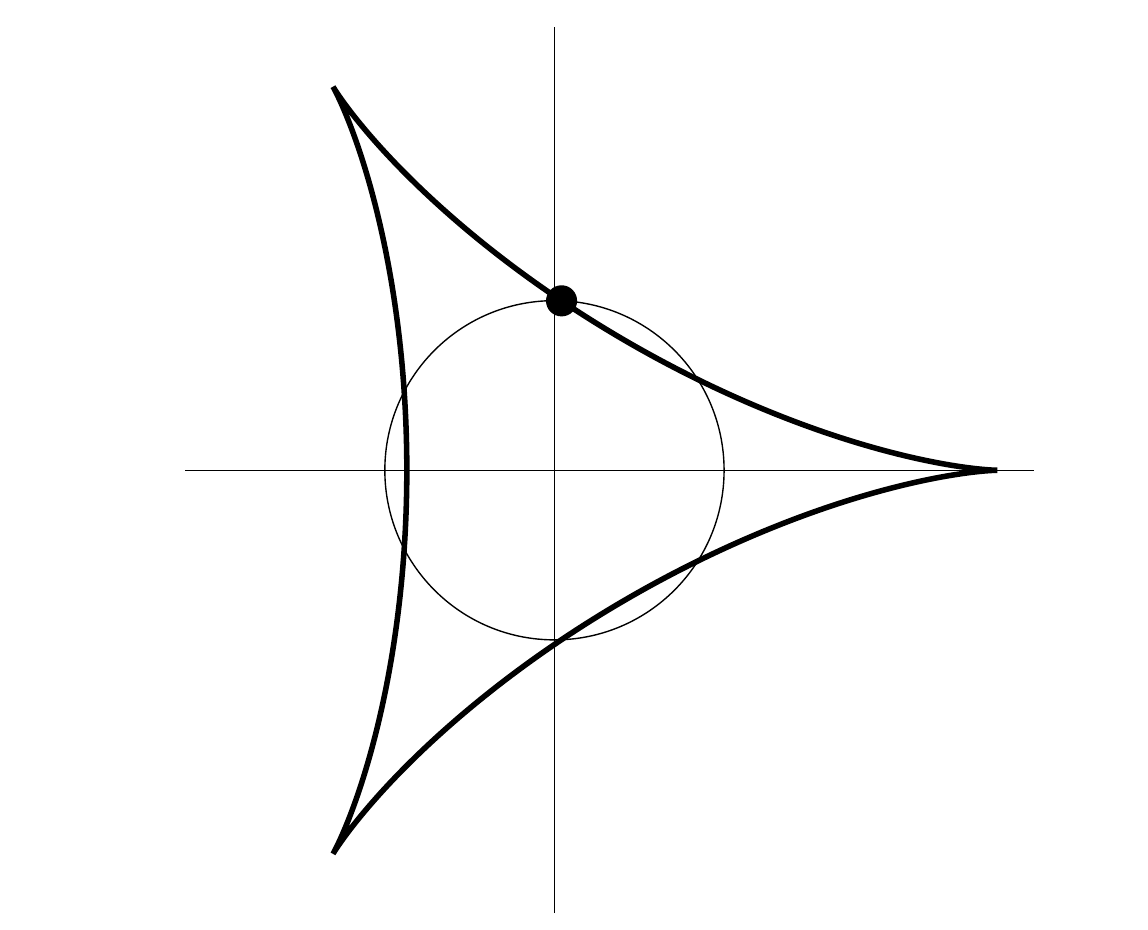}
\caption{The above diagram illustrates a circle of radius $\frac{\sqrt{33}}{5}$ intersecting a tricuspoid ($c=3$), to form one of the $5$-division points of the tricuspoid (the second intersection of the circle with the tricuspoid in the first quadrant is at the first $5$-division point). (Note that this point does not lie on the $y$-axis.)}
\end{figure}

\section{$n$-Division Points of the Tricuspoid} \label{notricuspoid}

In the previous section, we assumed that we had a rational hypocycloid to find the $n$-division points of already drawn on the page. In this section, we will consider the other problem mentioned in Remark~\ref{remark1}, namely that of constructing the $n$-division points of a hypocycloid $C$ without $C$ being drawn in advance. In doing so, we restrict to the case of a tricuspoid (hypocycloid where $a/b=c=3$). This necessitates being able to ascertain the location of the $n$-division points of the tricuspoid without intersecting lines or arcs with a tricuspoid; in other words, we need to be able to construct the rectangular coordinates of the $n$-division points of the tricuspoid solely with a straightedge and compass.

\begin{thm} The $n$-division points of the tricuspoid can be constructed with a straightedge and compass and no pre-drawn tricuspoid, if and only if $n\mid 6$. \end{thm}

\begin {proof}
We will prove this by showing that for $n\nmid 6$, there is some point $P_i$ in the set $\{P=P_1,P_2,\ldots,P_n\}$ of $n$-division points of the tricuspoid that is not constructible with a compass and straightedge

The tricuspoid can be parametrized in terms of the angle $\phi$ as \begin{align*} x(\phi) &= 2\cos(\phi)+\cos(2\phi), \\ y(\phi) &= 2\sin(\phi) - \sin(2\phi).\end{align*} We can compute the polar radius $r$ of any point on the curve in terms of $\phi$ using the above rectangular equations, and we can compute the polar radius $r$ of the $n$-division points of the tricuspoid in terms of the desired arc length fraction $s$ using the expression obtained in \S\ref{Hypocycloid}, with $c=3$. 

Equating these two expressions for the polar radius enables us to determine an expression for $\phi$ and subsequently for the $x$ coordinate; the first expression comes from evaluating the final result of the previous section when $c=3$ and the second is the parametric definition of $x$ for the tricuspoid in terms of $\phi$, the tangential angle: \[r^2=\frac{9n^2-96n+288}{n^2}=x^2+y^2=(2\cos \phi+\cos 2\phi)^2 + (2\sin \phi - \sin 2\phi)^2\] \[x=2\cos(\phi)+\cos(2\phi)\] Since $y = \sqrt{r^2-x^2}$, and the polar radius $r$ is constructible for all $n$-division points of the tricuspoid, the $x$ and $y$ coordinates of the $n$-division points of a tricuspoid are constructible if and only if the $x$ coordinate is constructible. 

Using the above equations, we find that the $x$-coordinate of the tricuspoid's first $n$-division point in terms of $n$ is a real root of the following cubic polynomial: \[f_n(x)=(2^2n^4)x^3-(3^3n^4-2^53^2n^3+2^53^3n^2)x - (3^3n^4-2^4 3^3n^3+2^4 3^2 17n^2-2^8 3^3n+2^7 3^4).\] Due to the parametrization described in \S \ref{Hypocycloid}, we require the first $n$-division point to fall within the first cusp of the tricuspoid, meaning that our polynomial is only valid for $n \geq 3$.

We will show that the $n$-division points are not constructible when $n \neq 3 \, \text{or} \, 6$ for $n \geq 3$.

Let us first consider the case where $(n,3)=1,\, n >2$. Then, we use the Newton polygon at $p=3$ to show that $f_n(x)$ is irreducible. Let us write $f_n(x)=a_3x^3+a_2x^2+a_1x+a_0$. (Here $a_2=0$.) For $a\in\QQ^\times$, let $v_3(a)$ denote the 3-adic valuation of $a$. Then \[v_3(a_0)=2,\ v_3(a_1)=2, \ v_3(a_3)=0.\] Thus, the Newton polygon, shown in Figure~\ref{newtonpoly} has a line with a slope of $-2/3$, which implies that $f_n(x)$ is irreducible.

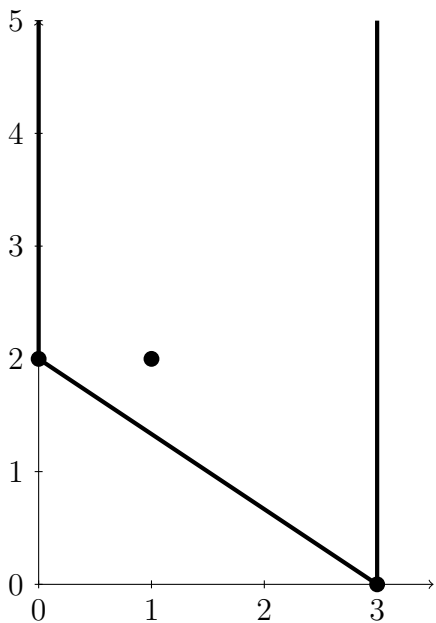
\begin{figure}
\begin{tikzpicture} [scale = 1.5]
  \draw[->] (0,0) -- (3.5,0);
  \draw[->] (0,0) -- (0,5);
  \foreach \x in {0,1,2,3}
    \draw (\x cm,1pt) -- (\x cm,-1pt) node[anchor=north] {$\x$};
  \foreach \y in {0,1,2,3,4,5}
    \draw (1pt,\y cm) -- (-1pt,\y cm) node[anchor=east] {$\y$};
  \draw[ultra thick] (0,5) -- (0,2);
  \draw[ultra thick] (3,5) -- (3,0);
  \draw[ultra thick] (0, 2) -- (3, 0);
  \fill (0, 2) circle (2pt);
  \fill (1, 2) circle (2pt);
  \fill (3, 0) circle (2pt);
\end{tikzpicture}
\caption{The Newton polygon for $f_n(x)$ at $p=3$, when $(n,3)=1$. As shown above, the polygon has a line with slope $-2/3$, indicating that $f_n(x)$ is irreducible.}
\label{newtonpoly}
\end{figure}


We now consider the case when $n$ is a multiple of $3$ and $n>6$. Suppose that $n=3^e m$ where $(m,3)=1$. Then the $m$-division points are contained in the $n$-division points, so it suffices to show that the $m$-division points are not constructible. Since $m$ is not a multiple of $3$, we see that unless $m=1$ or 2, the $m$-division points are not constructible. In all remaining cases, $n$ is a multiple of 9, so the 9-division points are included among the $n$-division points. However, the $9$-division points of a tricuspoid are not constructible, since $f_9(x)=81(324x^3 - 459x -107)$ is irreducible.

Finally, $f_n(x)$ is, in fact, reducible for $n=3$ and 6, so the coordinates of the 3- and 6-division points lie in a quadratic extension of $\QQ$ (in fact, $\QQ(\sqrt{3})$), which shows that they are constructible.
\end{proof}

\begin{rem}
Note that since the value of $x$ is the root of a cubic with rational coefficients, were we to allow ourselves an angle trisector or origami in addition to a compass or straightedge, we would be able to construct all of the $n$-division points of the tricuspoid even without it drawn.
\end{rem}

\begin{rem}
One may be tempted to wonder which, if any, of the \emph{individual} $n$-division points are constructible when $n\nmid 6$. It appears that no $n$-division points other than those contained among the 6-division points are constructible. The $x$-coordinates of $n$-division points are roots of the polynomials $f_r(x)$, where $r$ is a rational number $\ge 3$ whose numerator divides $n$. It appears that, except in the cases already discussed, $f_r(x)$ is indeed irreducible, and further, that the field $\QQ(\alpha)$ obtained by adjoining a root of $f_r(x)$ is totally ramified at 3. However, the analysis using Newton polygons in a na\"\i ve manner as above fails, as the slope of the Newton polygon is sometimes integral. This is due to the fact that (assuming $f_r$ is indeed irreducible) $\ZZ_3[\alpha]$ is not typically the full valuation ring of the field $\QQ(\alpha)\otimes\QQ_3$. \end{rem}



\bibliographystyle{alpha}
\bibliography{hypocycloid}

\end{document}